\documentclass{cedram-aif}

% ...

% ...

\title
[Representation fields]
{Representation fields for \\ commutative orders}

\alttitle{Corps de representations pour ordres commutatifs}

\author{\firstname{Luis} \middlename{} \lastname{Arenas-Carmona}}

\address{Universidad de Chile,\\
Departamento de matematicas,\\
Facultad de ciencia,\\
Casilla 653,\\
Santiago, Chile.}

\email{learenas@u.uchile.cl}

\thanks{Supported by Fondecyt, Grant 1100127.}


\altkeywords{ordres maximaux, alg\`ebres centrales simples, genre
spinoriel, corp de classes spinoriel}


\theoremstyle{plain}
\newtheorem{propi}{Proposition}[section]
\newtheorem{cor}{Corollary}
\newtheorem{lem}[propi]{Lemma}

\theoremstyle{definition}
\newtheorem{ex}[propi]{Example}

\newtheorem{kasse}{Case}

\newcommand\diag{\textnormal{diag}}
\newcommand\alge{\mathfrak{A}}
\newcommand\oink{\mathcal O}
\newcommand\Da{\mathfrak{D}}
\newcommand\Ha{\mathfrak{H}}
\newcommand\matrici{\mathbb{M}}
\newcommand\finitum{\mathbb{F}}
\newcommand\ad{\mathbb{A}}
\newcommand\Txi{\lceil}

\begin{document}
\begin{abstract}
A representation field for a non-maximal order $\Ha$ in a central
simple algebra is a subfield of the spinor class field of maximal
orders which determines the set of spinor genera of maximal orders
containing a copy of $\Ha$. Not every non-maximal order has a
representation field. In this work we prove that every commutative
order has a representation field and give a formula for it. The
main result is proved for central simple algebras over arbitrary
global fields.
\end{abstract}

\begin{altabstract}
Un corps de repr\'esentation pour un ordre non maximal $\Ha$ dans
une alg\`ebre centrale simple est un sous-corps du corps de
classes espinoriel d'ordres maximaux qui d\'etermine l'ensemble de
genres espinoriels d'ordres maximaux qui contiennent un conjugu\'e
de $\Ha$. Un ordre non maximal ne poss\`ede pas forcement un corps
de repr\'esentation. Dans ce travail nous montrons que chaque
ordre commutatif a un corps de repr\'esentation $F$ et nous
donnons une formule pour $F$. Le r\'esultat principal est prouv\'e
pour des alg\`ebres simples centrales sur des corps globaux
arbitraires.
\end{altabstract}

\maketitle

\section{Introduction}
Let $K$ be a number field. Let $\alge$ be either, a central simple
$K$-algebra of degree $n>2$, or a quaternion algebra satisfying
Eichler condition \cite{FriedmannQ}. The set $\mathbb{O}$ of
maximal orders in $\alge$ can be split into isomorphism classes
or, equivalently, conjugacy classes \cite{reiner}. Let
$\overline{\mathbb{O}}$ be the set of isomorphism classes. In
\cite{spinor} we found a field extension $\Sigma/K$ whose Galois
group $\mathrm{Gal}(\Sigma/K)$ yields a natural free and
transitive action on $\overline{\mathbb{O}}$. For some families of
orders $\Ha$ in $\alge$ we found a field $F=F(\Ha)$ between $K$
and $\Sigma$ such that the set $\overline{\mathbb{O}}_{\Ha}$
 of classes of orders containing
a conjugate of $\Ha$ is exactly one orbit of the subgroup
$\mathrm{Gal}(\Sigma/F)$. In particular, the number of isomorphism
classes of such orders is $[\Sigma:F]$. The first result in this
direction seems to be due to Chevalley who considered the case
when $\alge$ is a matrix algebra and $\Ha$ is the maximal order of
a maximal subfield $L$ \cite{Chevalley}. In our notation,
Chevalley result implies $F(\Ha)=L\cap\Sigma$. In \cite{spinor} we
extended this result to algebras $\alge$ with no partial
ramification.
  Chinburg and Friedman have given a similar result
in which $\Ha$ is an arbitrary rank-two commutative order in a
quaternion algebra $\alge$ satisfying Eichler condition
\cite{FriedmannQ}. In case $L=K\Ha$ is a field, the representation
field $F$ is a subfield of $L\cap\Sigma$ that depends on the
relative discriminant of the order $\Ha$. A more recent article by
Linowitz and Shemanske \cite{Linowitz} extends this result to
central simple algebras of prime degree. Here we extend these
results by proving the existence and providing a formula for the
representation field for every commutative order $\Ha$ in a
central simple algebra $\alge$ of arbitrary degree.

We use the language of spinor genera in all of this work for the sake of generality.
 In this general setting, the set $\overline{\mathbb{O}}$ is not the set of
 conjugacy classes but the set of spinor genera of maximal orders, while
  $\overline{\mathbb{O}}_{\Ha}$ is the set of spinor genera $\Phi$ with at least one order
  $\Da\in\Phi$ containing $\Ha$ as a sub-order. All of the above extends to this setting.
  A single spinor genera can have more than one class only
when $\alge$ is a quaternion algebra failing to satisfy Eichler
condition \cite{spinor}.

\begin{thm}\label{T1}
Let $\alge$ be a central simple algebra over the number field $K$.
Let $\Ha$ be an arbitrary commutative order. Let $\Da$ be a
maximal order containing $\Ha$. For every maximal ideal $\wp$ in
the ring of integers $\oink_K$ we let $I_\wp$ be the only maximal
two-sided ideal of $\Da$ containing $\wp1_{\Da}$, and let
$\mathbb{H}_\wp$ be the image of $\Ha$ in $\Da/I_\wp$. Let
$\mathbb{E}_\wp$ be the center of the ring $\Da/I_\wp$. Let
$t_\wp$ be the greatest common divisor of the dimensions of the
irreducible $\mathbb{E}_\wp$-representations of the algebra
$\mathbb{E}_\wp\mathbb{H}_\wp$. Then the representation field
$F(\Ha)$ is the maximal subfield $F$, of the spinor class field
$\Sigma$, such that the inertia degree $f_\wp(F/K)$ divides
$t_\wp$ for every place $\wp$.
\end{thm}

In the above statement, it is implied that $\mathbb{E}_\wp$ is a
field. In fact, $\mathbb{E}_\wp$ is the residue field of a local
division algebra $E_\wp$ such that
$\alge_\wp\cong\matrici_m(E_\wp)$.
 We show in \S5 that Theorem \ref{T1} is indeed a
 generalization of the results in \cite{spinor}, \cite{FriedmannQ},
 and \cite{Linowitz}. For the sake of generality, we prove a generalization of Theorem
 \ref{T1} that includes also the function field
case studied in \cite{sheaves}. In order to
 consider both cases simultaneously, we introduce the concept of
 A-curve in \S2.

\section{A-curves and Spinor class fields}

In all that follows,
  we say that $X$ is an A-curve with field of
functions $K$, in any of the following cases:
\begin{kasse}\label{kasse1} $X$ is a smooth
irreducible curve over a finite field $\finitum$ and $K$ is the
field of rational functions on $X$.\end{kasse}
\begin{kasse}\label{kasse2}
$X=\mathop{\mathrm{Spec}}\oink_S$ where $\oink_S$ is the ring of
$S$-integers for a non-empty finite set $S$ of places over a
global field $K$ containing the archimedean places.\end{kasse} In
both cases $X$ is provided with its Zariski Topology and its
structure sheaf $\oink_X$. Furthermore, in Case \ref{kasse2} every
$S$-lattice $\Lambda_0$ in a $K$-vector space $V_K$ defines a
locally free sheaf $\Lambda$ by $\Lambda(U)=\oink_X(U)\Lambda_0$,
and the stalk at a place $\wp\in X$ is just the free
$\oink_{[\wp]}$-lattice $\Lambda_{0[\wp]}$, where $\oink_{[\wp]}$
and $\Lambda_{0[\wp]}$ denote the localizations at $\wp$ (as
opposed to completions). The sheaf $\Lambda$ thus defined is
called an $X$-lattice.
 Analogously, we define an $X$-lattice in Case \ref{kasse1} as a locally free
 sheaf of lattices as in \cite{sheaves}. The space $W$ spanned by
 $\Lambda(U)$ is independent
of the affine set $U$. We call $W$ the space generated by
$\Lambda$ and denote it $K\Lambda$. The affine sub-case in Case
\ref{kasse1} is also a sub-case of Case \ref{kasse2}, if we let
$S$ be the complement of $X$ in its smooth projective completion,
and both definitions of $X$-lattice coincide in this case. We say
that $S=\emptyset$ when $X$ is a projective curve over a finite
field.

 Note that the completion $\Lambda_\wp$ of the sheaf $\Lambda$ at a place $\wp\in X$ is
defined as the completion of the stalk $\Lambda_{[\wp]}$, or equivalently the
completion of the lattice $\Lambda(U)$ for any affine set $U\subseteq X$ containing
$\wp$. Since $X$-lattices can be defined by pasting lattices defined on the sets of
an affine cover, some properties of usual lattices (\cite{Om} \S81) are inherited by
$X$-lattices, namely:
\begin{itemize}
\item A lattice $\Lambda$ is determined by it set of completions
$\{\Lambda_\wp\}_\wp$, and it can be modified at a finite number of places to get a
new lattice . \item The adelization $\mathrm{GL}_\ad(V)$ of the group
$\mathrm{GL}(V)$ of $K$-linear maps on $V$ acts on the set of lattices by acting on
the set of completions at every place $\wp$. \item If $V=\alge$ is an algebra, an
$X$-lattice $\Da$ is an order (i.e., a sheaf of orders) if and only if every
completion is an order. The same holds for maximal orders\footnote{This is not in
the reference, but follows easily from the previous results, since a lattice $\Da$
is an order if and only if $1\in\Da$ and $\Da\Da=\Da$.}.
\end{itemize}
 Since any two local maximal orders are conjugate in any
completion $\alge_\wp$ with $\wp\in X$, it follows that for any
two global maximal orders $\Da$ and $\Da'$ there exist an element
$a$ in the adelization $\alge^*_{\mathbb{A}}$ satisfying
$\Da_{\mathbb{A}}'=a\Da_{\mathbb{A}} a^{-1}$, where
$$\Da_{\mathbb{A}}=\prod_{\wp\in S}\alge_\wp\times\prod_{\wp\in X}\Da_\wp$$
and $\Da'_{\mathbb{A}}$ is defined analogously. As usual we write
simply $\Da'=a\Da a^{-1}$, since no other action of adelic points
on orders is considered in this work.
 The orders $\Da$ and $\Da'$ are conjugate if we can choose
$a\in\alge^*$. We say that these orders are in the same spinor
genus if we can choose $a=bc$ where $b\in\alge^*$ and $c_\wp$ has
reduced norm $1$ for every place $\wp\in X$. The following result
is needed in what follows. It is proved in \cite{spinor} for Case
\ref{kasse1} and follows easily from \cite{reiner} (Theorem 33.4
 and  Theorem 33.15), or remark 2.5 in \cite{sheaves},
  for Case \ref{kasse2}.
\begin{quote}
 The set of spinor genera of maximal orders in $\alge$ is in
correspondence with the group $J_K/K^*H(\Da)$, where $J_K$ is the idele group of $K$
and $H(\Da)\subseteq J_K$ is the image under the reduced norm of the
conjugation-stabilizer $(\alge^*_{\mathbb{A}})^{\Da}$.
\end{quote}
By the strong approximation theorem, two orders in the same spinor
genus are always conjugate if the automorphism group of $\alge$ is
non-compact at some place $\wp\in S$. When $K$ is a number field
and $S=\infty$, this condition is equivalent to Eichler condition
if $\alge$ is a quaternion algebra, and it is always satisfied for
algebras of higher dimension. When $X$ is a projective curve we
have $S=\emptyset$, so the condition cannot hold. Note however
that spinor genera of orders over projective curves still carry
some global information that can be recovered in any affine subset
of $X$ \cite{sheaves}. When $X$ is an affine curve, the condition
holds unless every infinite place of $X$ is totally ramified for
$\alge/K$. Since we express our main result in terms of spinor
genera, this result holds in full generality, but it is important
to keep in mind the previous remark in the applications.

 We let $\Sigma$ be the class
field corresponding to the group $K^*H(\Da)$ (\cite{weil}, \S
XIII.9). For every $\rho\in\mathrm{Gal}(\Sigma/K)$, let
$a_\rho\in\alge_{\ad}^*$ be any element satisfying
$\rho=[n(a_\rho),\Sigma/K]$, where $t\mapsto[t,\Sigma/K]$ denotes
the Artin maps on ideles and $n:\alge^*_{\ad}\rightarrow J_K$
denotes the reduced norm.
The action of $\mathrm{Gal}(\Sigma/K)$ on the set
$\overline{\mathbb{O}}$ of spinor genera of maximal orders is
given by $$\rho.\mathrm{spin}(\Da)=\mathrm{spin}(a_\rho\Da
a_\rho^{-1}),\qquad\forall(\rho,\Da)\in\mathrm{Gal}(\Sigma/K)\times\mathbb{O}.$$

Assume in all that follows that $\Ha$ is a suborder of $\Da$. A generator for
$\Da|\Ha$ is an element $u\in\alge^*_{\ad}$ such that $\Ha\subseteq u\Da u^{-1}$.
Let $\Txi(\Da|\Ha)$ the set of reduced norms of generators. There is no reason a
priori for this set to be a group. Note that if $a$ normalizes $\Da$ and $b$
normalizes $\Ha$, then for any generator $u$ the element $bua$ is a generator. In
particular $H(\Da)\Txi(\Da|\Ha)=\Txi(\Da|\Ha)$.
 We conclude that the set
$K^*\Txi(\Da|\Ha)$ is completely determined by its image
$\mathcal{G}(\Da|\Ha)\subseteq\mathrm{Gal}(\Sigma/K)$ under the
Artin map. If $\mathcal{G}(\Da|\Ha)$ is a group we define
 $$F(\Ha)=\big\{a\in\Sigma|g(a)=a\ \forall
 g\in\mathcal{G}(\Da|\Ha)\big\},$$ and call $F=F(\Ha)$ the
 representation field for $\Ha$. It follows from the definition
 that an order $\Da'=b\Da b^{-1}$, for $b\in\alge_\ad^*$, is in the
 same spinor genus than some maximal order containing a conjugate of
 $\Ha$, if and only if $n(b)\in n(u)K^*H(\Da)$ for some generator
 $u$ for $\Da|\Ha$, or equivalently $[n(b),F/K]=\mathrm{id}_F$. We conclude
  that $F$ satisfies the properties stated in \S1. If
 $\mathcal{G}(\Da|\Ha)$ fails to be a group we say that the
 representation field for $\Ha$ is not defined. For examples of
 non-commutative orders for which the representation field is not
 defined, see \cite{counter} or Example 3.6 below.

\section{Spinor image and representations}

Let $K$ be a local field of arbitrary characteristic.
 Assume in all of this
section that $\alge=\mathbb{M}_m(E)$ where $E$ is a central division algebra over
$K$ with ring of integers $\oink_E$, maximal ideal $\mathfrak{m}_E$,
 and residue field $\mathbb{E}=\oink_E/\mathfrak{m}_E$. Let $\pi_E$ and
$\pi_K$ be the uniformizing parameters of $E$ and $K$ respectively. Recall that if
$n:\alge^*\rightarrow K^*$ is the reduced norm, then $n(\pi_E1_{\alge})=v\pi_K^m$
 for some unit $v\in\oink_K^*$.

Let $\Ha\subseteq\Da$ be orders in $\alge$. Then a local generator
$u$, or simply a generator whenever confusion is unlikely, is an
element $u\in\alge^*$ satisfying $\Ha\subseteq u\Da u^{-1}$. In
the rest of this section, we assume that $\Da$ is maximal, but we
make no assumption on $\Ha$ for the sake of generality.
 The purpose of this section is to characterize the set of
reduced norms of local generators in terms of the degrees of the
representations of the reduction of $\Ha$ m\'odulo $\pi_E\Da$.

For short, we let $\stackrel{\rightarrow}1_q$ and $\stackrel{\rightarrow}0_q$ denote
vectors $(1,\dots,1)$ and $(0,\dots,0)$ of length $q$. For example
$(\stackrel{\rightarrow}1_3,2\cdot
\stackrel{\rightarrow}1_2,\stackrel{\rightarrow}0_3)=(1,1,1,2,2,0,0,0)$. For any
ring $A$ we denote by $A^m$ the free $A$-module with $m$ generators, while $A^{*m}$
denotes the set of $m$-powers of invertible elements in $A$.

\begin{lem}\label{l33}
Let $\Ha$ be a suborder of the maximal order $\Da=\matrici_m(\oink_E)$, and let
$\mathbb{H}$ be its image in $\matrici_m(\mathbb{E})$.  Assume that the reduced norm
of a generator $u$ for $\Da|\Ha$ spans the principal ideal
$\Big(n(u)\Big)=(\pi_K^s)=\pi_K^s\oink_K$. Then there exists a flag of
$\mathbb{H}$-modules
$$\{0\}=V_0\subseteq V_1\subseteq\cdots\subseteq V_t=\mathbb{E}^m$$
 and integers $r_1,\dots,r_t$ satisfying $$s=\sum_{i=1}^tr_i\dim_{\mathbb{E}}(V_i/V_{i-1}).$$
\end{lem}

\begin{proof}
 Let $P,Q\in\matrici_m(\oink_E)^*$ be such that $u=PDQ$ for a diagonal matrix
$D=\mathrm{diag}(\pi_E^{r_1}\stackrel{\rightarrow}1_{q_1},\dots,
\pi_E^{r_t}\stackrel{\rightarrow}1_{q_t})$, where $r_1<r_2<\cdots<r_t$ and
$q_1+\cdots+q_t=m$. The condition $\Ha\subseteq u\Da u^{-1}$ is equivalent to
$\Ha(u\oink_E^m)\subseteq u\oink_E^m$, or equivalently
$$(P^{-1}\Ha P)D\oink_E^m\subseteq D\oink_E^m.$$ Replacing $\Ha$ by
$P^{-1}\Ha P$ if needed, we can assume that $P$ is the identity. Now we define
$V_i\subseteq\mathbb{E}^n$ as the image of the matrix
$$\mathrm{diag}(\stackrel{\rightarrow}1_{q_1},\dots,
\stackrel{\rightarrow}1_{q_i},\stackrel{\rightarrow}0_{q_{i+1}},\dots,
\stackrel{\rightarrow}0_{q_t})\in\matrici_m(\mathbb{E}).$$
Certainly the flag $V_0\subseteq\dots\subseteq V_t$ satisfies the
last condition. We claim that every $V_i$ is a submodule.
Otherwise, there exists an element $\overline{v}\in V_i$, for some
$i$, and an element $\overline{h}\in \mathbb{H}$ such that
$\overline{h}\overline{v}\notin V_i$. Now take pre-images $h\in
\Ha$ and $v\in\widetilde{V}_i$ of $\overline{h}$ and
$\overline{v}$ respectively, where
$\widetilde{V}_i\subseteq\oink_E^m$ is the image of the matrix
$$\mathrm{diag}(\stackrel{\rightarrow}1_{q_1},\dots,
\stackrel{\rightarrow}1_{q_i},\stackrel{\rightarrow}0_{q_{i+1}},\dots,
\stackrel{\rightarrow}0_{q_t})\in\matrici_m(\oink_E).$$ Then $\pi_E^{r_i}v\in
D\oink_E^m$, whence $h(\pi_E^{r_i}v)\in D\oink_E^m$. Since conjugation by $\pi_E$
leaves $\Da$ invariant, then $\pi_E^{-r_i}h\pi_E^{r_i}$ is an integral matrix with
reduction $\hat h$. Since $h(\pi_E^{r_i}v)\in D\oink_E^m$, then
$$\pi_E^{-r_i}h(\pi_E^{r_i}v)\in
\pi_E^{-r_i}(D\oink_E^m)\cap \oink_E^m\subseteq \widetilde{V}_i+\pi_E\oink_E^m,$$
whence
 $\hat h(\overline{v})\in V_i$. Since
$V_i$ is generated by $V_{i\mathbb{K}}=V_i\cap\mathbb{K}^m$, where $\mathbb{K}$ is
the residue field of $K$, we may assume that $\overline{v}\in\mathbb{K}^m$. Note
that $\hat h$ is the image of $\overline h$ under an automorphism
$\sigma\in\mathrm{Gal}(\mathbb{E}/\mathbb{K})$, and since
$\overline{v}\in\mathbb{K}^m$, we have $\hat{h}(\overline{v})=\sigma[\overline
h(\overline{v})]$. Since $\sigma$ fixes $V_{i\mathbb{K}}$ point-wise, then
necessarily $\sigma(V_i)=V_i$.
 It follows that $\hat h(\overline{v})\in V_i$ if and only if
$\overline h(\overline{v})\in V_i$. This contradicts the choice of $\overline{v}$.
\end{proof}

\begin{lem}\label{l33b}
Let $\Ha$ be a suborder of the maximal order $\Da=\matrici_m(\oink_E)$, and let
$\mathbb{H}$ be its image in $\matrici_m(\mathbb{E})$. Assume that there is an
$\mathbb{H}$-module $V\subseteq\mathbb{E}^m$ of dimension $r$. Then there exists a
local generator $u$ for $\Da|\Ha$ whose reduced norm generates the ideal
$(\pi_K^{m-r})$.
\end{lem}

\begin{proof}
 Let $M\in\matrici_m(\mathbb{E})$ be a matrix of rank $r$.
We claim that $M$ has a pre-image $\tilde{M}$ in $\matrici_m(\oink_E)$, satisfying
the following conditions:
\begin{itemize}
\item $\tilde{M}$ is invertible in $\matrici_m(E)$.\item The
reduced norm $n(\tilde{M})$ is a generator of the ideal
$(\pi_K^{m-r})$. \item $\pi_E\tilde{M}^{-1}\in\matrici_m(\oink_E)$
and its reduction $L\in\matrici_m(\mathbb{E})$ satisfies
$\textrm{Im}(M)=\textrm{ker}(L)$.
\end{itemize}
To define the lifting $\tilde{M}$ we write  $M=PDQ$, where $P$ and $Q$ are
invertible in $\matrici_m(\mathbb{E})$ and
$D=\mathrm{diag}(\stackrel{\rightarrow}1_r,\stackrel{\rightarrow}0_{m-r})$,
 then chose arbitrary liftings $\tilde{P}$ and $\tilde{Q}$ in $\matrici_m(\oink_E)^*$,
  and finally set $\tilde{M}=\tilde{P}\tilde{D}\tilde{Q}$,
 where $\tilde{D}=\mathrm{diag}(\stackrel{\rightarrow}1_r,
 \pi_E\stackrel{\rightarrow}1_{m-r})\in\matrici_m(\oink_E)$.  Then
$$\pi_E\tilde{M}^{-1}=(\pi_E\tilde{Q}^{-1}\pi_E^{-1})\Big(
\mathrm{diag}(\pi_E\stackrel{\rightarrow}1_r,\stackrel{\rightarrow}1_{m-r})\Big)
\tilde{P}^{-1}\in\matrici_m(\oink_E).$$ Note that the matrix
$\pi_E\tilde{Q}^{-1}\pi_E^{-1}$ is invertible in $\matrici_m(\oink_E)$, whence the
reduction $L$ of $\pi_E\tilde{M}^{-1}$ has the right rank. The last condition
follows now if we prove $LM=0$, but this is clear since it is the reduction of
$\pi_E\tilde{M}^{-1}\tilde{M}=\pi_E\mathrm{Id}$.

Let $M$ be a projection matrix with image $V$ and let $L$ be as above. Then $M$ has
rank $r$ and satisfies $LXM=0$ for every $X\in\mathbb{H}$, since $X$ leaves
$V=\textrm{Im}(M)$ invariant. We claim that $\tilde{M}$ is a generator for
$\Da|\Ha$, this concludes the proof. Since every element $X\in\mathbb{H}$ leaves the
$\mathbb{E}$-vector space $V=M(\mathbb{E}^m)$ invariant, then every element $x$ of
$\Ha$ must leave invariant its pre-image, i.e., the module
$\tilde{M}(\oink_E^m)+\pi_E\oink_E^m$. Now we have
$$\pi_E\oink_E^m=\tilde{M}(\tilde{M}^{-1}\pi_E\oink_E^m)\subseteq\tilde{M}\oink_E^m,$$
since
$\tilde{M}^{-1}\pi_E=\pi_E^{-1}(\pi_E\tilde{M}^{-1})\pi_E\in\matrici_m(\oink_E)$.
As this implies $x(\tilde{M}\oink_E^m)\subseteq\tilde{M}\oink_E^m$
for any $x\in\Ha$, we conclude that $\tilde{M}$ is a generator.
The result follows.\end{proof}

\begin{lem}\label{l34}
Let $\Ha$ be a suborder of the maximal order $\Da=\matrici_m(\oink_E)$ and let
$\mathbb{H}$ be its image in $\matrici_m(\mathbb{E})$. Assume that every irreducible
representation of the $\mathbb{E}$-algebra $\mathbb{E}\mathbb{H}$ has dimension $d$.
 Then there exists a local generator for
$\Da|\Ha$ whose reduced norm generates $(\pi_K^{s})$ if and only if $d$ divides $s$.
\end{lem}

\begin{proof}
 On one hand, assume $d$ divides $s$. Since conjugation by any power of $\pi_E$ leaves
 $\Da$ invariant and the reduced norm of the diagonal matrix $\pi_E1_{\alge}$ spans the
 ideal $(\pi_K^m)$, we may assume $0\leq s<m$. It suffices to show that there
exists a $\mathbb{H}$-submodule of $\mathbb{E}^m$ of dimension $m-s$. By the
representation theory for finite dimensional algebras, there exists a composition
series
$$\{0\}=V_0\subseteq V_1\subseteq\cdots\subseteq V_{m/d}=\mathbb{E}^m$$ where the submodule
$V_i$ has dimension $di$, and the result follows. On the other
hand, if there exists a local generator for $\Da|\Ha$ whose
reduced norm spans $(\pi_K^{s})$, then there exists a sequence of
submodules $V_1,\dots,V_t$ and integers $r_1,\dots,r_{t}$ as in
Lemma \ref{l33} satisfying
$s=\sum_ir_i\dim_{\mathbb{E}}(V_i/V_{i-1})$. Refining the sequence
if needed, we may assume it is a composition series, whence
$\dim_{\mathbb{E}}(V_i/V_{i-1})=d$ for every $i$, and therefore
$d$ divides $s$.
\end{proof}

Next result is now immediate from the previous lemma:

\begin{lem}\label{lemma34}
 Let $\Ha$ be a suborder of
$\Da=\mathbb{M}_m(\oink_E)$ and let $\mathbb{H}$ be its image in
$\mathbb{M}_m(\mathbb{E})$. Assume that any irreducible representation of the
$\mathbb{E}$-algebra $\mathbb{E}\mathbb{H}$ has the same degree $d$. Then the set of
reduced norms of generators for $\Da|\Ha$ is $\Txi(\Da|\Ha)=K^{*d}\oink_K^*$.\qed
\end{lem}

The results in this section seem to indicate that the existence of a representation
field depends only on the dimensions of the representations of the algebra
$\mathbb{E}\mathbb{H}$. This is not so, as illustrated by Example \ref{counter}. It
holds, however, when $\Ha$ is the largest order with reduction $\mathbb{H}$ as next
lemma shows.

\begin{lem}\label{bigorder} Let $\Ha=p^{-1}(\mathbb{H})$ be a suborder of
$\Da=\mathbb{M}_m(\oink_E)$, where $\mathbb{H}$ is a subalgebra of
$\mathbb{M}_m(\mathbb{E})$ and
$p:\mathbb{M}_m(\oink_E)\rightarrow\mathbb{M}_m(\mathbb{E})$ the usual projection.
Then there exists a local generator for $\Da|\Ha$ whose reduced norm generates
$(\pi_K^{s})$ if and only if there exists an $\mathbb{H}$-module
$V\subseteq\mathbb{E}^m$ of dimension $t$, where $m-t$ is the remainder of $s$ when
divided by $m$.
\end{lem}

\begin{proof}
 It suffices to prove that if $u$ is a generator whose reduced
norm generates $(\pi_K^{s})$ then there exists an $\mathbb{H}$-module $V$ of
dimension $t$. Since $u$ is a generator, we have $\Ha(u\oink_E^m)\subseteq
u\oink_E^m$. Post-multiplying $u$ by a power of $\pi_E$ we may assume
$u\oink_E^m\subseteq\oink_E^m$, but $u\oink_E^m$ is not contained in
$\pi_E\oink_E^m$. Since $\Ha$ contains every matrix of the form $\pi_Ea$ with
$a\in\Da$, we conclude that $\pi_E\oink_E^m\subseteq u\oink_E^m$. It follows that
$u=PDQ$ where $P$ and $Q$ are units in $\matrici_m(\oink_E)$ and
$D=\diag(\pi_E\stackrel{\rightarrow}1_r,\stackrel{\rightarrow}1_d)$. Let $V$ be the
image of $u\oink_E^m$ in $\mathbb{E}^m$.  It follows that $\dim_{\mathbb{E}}V=d$
while on the other hand the reduced norm of $u$ generates $(\pi_K^r)$, so that
$r=s$. It follows that $d=m-r=t$, since they are congruent modulo $m$ and $0<d,t\leq
m$.
\end{proof}

\begin{ex}\label{counter} Let $\alge=\matrici_n(K)$ and let
$L=\{a\in\alge|ae_1\in Ke_1\}$, where $\{e_1,\dots,e_n\}$ is the
canonical basis of $K^n$. Let $I$ be an integral ideal, i.e.,
 a sub-lattice of $\oink_X$,
and let $\Ha=\Ha(I)$ be the order
$$\Ha(I)=\Big(L\cap\matrici_n(\oink_X)\Big)+\matrici_n(I),$$
where $\matrici_n(I)$ is the sheaf defined by
$\matrici_n(I)(U)=\matrici_n\Big(I(U)\Big)$ for any affine set $U$ and
$\matrici_n(\oink_X)$ is defined analogously.

The set of reductions $\mathbb{H}_\wp$ depends only on the primes
$\wp$ such that $I_\wp\neq\oink_\wp$, but this is not so for the
set of norms of generators. In fact, if $I_\wp$ is the maximal
ideal of $\oink_\wp$, then Lemma \ref{bigorder} applies and any
local generator at $\wp$ must have reduced norm in $\oink^*_\wp
K_\wp^{*n}$ or $\pi_\wp^{n-1}\oink^*_\wp K_\wp^{*n}$. In
particular, if $I=\wp$ and if $[\wp,\Sigma/K]$ has order at least
$3$ in the Galois group $\mathrm{Gal}(\Sigma/K)$, then the
representation field does not exist.  However, if
$I_\wp=(\pi_\wp^t)$, then
$\diag\left(\pi_\wp^{-s},\stackrel{\rightarrow}1_{n-1}\right)=\diag(\pi_\wp^{-s},1,\dots,1)$
is a generator for any $s\leq t$. In particular, if $t=t(\wp)$ is
big enough for every prime $\wp$ dividing $I$, then the
representation fields is defined.
\end{ex}

\section{Proof of Theorem \ref{T1}}

\begin{lem}\label{l2} Let $R$ be a commutative ring that is complete with respect to an ideal
$J$. Then any idempotent $P\in R/J$ can be lifted to an idempotent
$\tilde{P}\in R$.
\end{lem}

\begin{proof}
Let $P_1$ an arbitrary lifting of $P$ to $R$ and define recursively
$P_{n+1}=P_n+(1-2P_n)(P_n^2-P_n)$. Note that
$$P_{n+1}^2-P_{n+1}=(P_n^2-P_n)\Big[1-(1-2P_n)^2[1-(P_n^2-P_n)]\Big].$$
Since $(1-2P_n)^2=1+4(P_n^2-P_n)\equiv 1$ (modulo $J$), we prove
by recursion that $P_n^2-P_n\in J^n$ and the result follows by
completeness.
\end{proof}

\begin{lem}\label{l3} Let $\Ha$ be an commutative order in a local central
simple algebra $\alge$ and let $\mathfrak{P}$ be the sub-order
generated by the idempotents of $\Ha$. Let $C$ be the centralizer
of $\mathfrak{P}$ in $\alge$. Write $\Ha$ as a direct product
$\Ha=\prod_{i\in I}\Ha_i$ with $\Ha_i$ connected. Then we have a
decomposition $C=\prod_{i\in I}C_i$, where every $C_i$ is a
central simple algebra. Furthermore, there exist a family of
maximal orders $\{\Da_i\}_{i\in I}$, with
$\Ha_i\subseteq\Da_i\subseteq C_i$ and a maximal order
$\Da\subseteq\alge$ such that$$\Txi(\Da|\Ha)\supseteq\prod_{i\in
I}\Txi(\Da_i|\Ha_i).$$
\end{lem}

\begin{proof}
If $\Ha=\prod_{i\in I}\Ha_i$ as above, we let $P_i$ be the
idempotent corresponding to the factor $\Ha_i$. Then $P_i$ is a
central idempotent in $C$, whence there exists a decomposition of
$C$ as above where $C_i\cong P_iC$. Assume $\alge=\matrici_m(E)$,
where $E$ is a division algebra. Note That $E^m=\prod_iP_iE^m$ as
either left $C$-modules or right $E$-modules. By taking an
$E$-basis of each $P_iE^m$, we obtain a basis of $E^m$ where the
order $\mathfrak{P}$ is the ring of diagonal matrices of the form
$$\left(\begin{array}{cccc}a_1I_{q_1}&0&\cdots&0\\0&a_2I_{q_2}&\cdots&0\\
\vdots&\vdots&\ddots&\vdots\\0&0&\vdots&a_rI_{q_r}\end{array}\right),\qquad
a_1,\dots,a_r\in \oink_K,$$
where $I_q$ is the $q$ by $q$ identity matrix.
 It follows that $C$ is the
algebra of matrices of the form
$$A=\left(\begin{array}{cccc}A_1&0&\cdots&0\\0&A_2&\cdots&0\\
\vdots&\vdots&\ddots&\vdots\\0&0&\vdots&A_r\end{array}\right),
\qquad A_i\in\matrici_{q_i}(E).$$ We can identify $C_i\cong
\matrici_{q_i}(E)$ with the set of matrices $A$, as above, such
that $A_j=0$ for $j\neq i$. In particular $C_i$ is central simple.
  By a suitable change of basis in every $P_iE^m$,
we can assume $\Ha_i\subseteq \matrici_{q_i}(\oink_E)$.
 Now we can choose $\Da=\matrici_m(\oink_E)$ and $\Da_i=\matrici_{q_i}(\oink_E)$ for $i=1,\dots,r$.

 Let $u_i\in C_i^*$ be a generator for
$\Da_i|\Ha_i$ as defined in \S2, for every $i$. Then
$u=\sum_iu_i\in C^*$ is a generator for $\Da|\Ha$, since
$$\Ha=\prod_i\Ha_i\subseteq\prod_i(u_i\Da_iu_i^{-1})=u\left(\prod_i\Da_i\right)u^{-1}\subseteq u\Da
u^{-1}.$$ Furthermore, from the explicit description of $C$ given above we see that
the reduced norm of $u$ is $n(u)=\prod_in(u_i)$. Now the conclusion
follows.\end{proof}

Now we are ready to prove the following generalization of Theorem \ref{T1}.

\begin{propi}
Let $X$ be an A-curve and let $K$ be the field of functions of $X$. Let $\alge$ be a
central simple $K$-algebra. Let $\Ha$ be an arbitrary commutative order in $\alge$.
Let $\Da$ be a maximal order containing $\Ha$. For every place $\wp\in X$ we let
$I_\wp$ be the only maximal two-sided ideal of the completion $\Da_\wp$, and let
$\mathbb{H}_\wp$ be the image of $\Ha$ in $\Da_\wp/I_\wp$. Let $\mathbb{E}_\wp$ be
the center of the ring $\Da_\wp/I_\wp$. Let $t_\wp$ be the greatest common divisor
of the dimensions of the irreducible $\mathbb{E}_\wp$-representations of the algebra
$\mathbb{E}_\wp\mathbb{H}_\wp$. Then the representation field $F(\Ha)$ is the
maximal subfield $F$ of the spinor class field $\Sigma$ such that the inertia degree
$f_\wp(F/K)$ divides $t_\wp$, for every place $\wp$.
\end{propi}

\begin{proof}
 We let
$\overline{\mathbb{H}}_\wp$ be the quotient of $\mathbb{H}_\wp$ by its radical. In
particular, $\overline{\mathbb{H}}_\wp=\Ha_\wp/J$ for an ideal $J$ such that
$J^n\subseteq \pi_E\Da$ for some $n$. It follows that $\Ha_\wp$ is complete with
respect to $J$, so any idempotent in $\overline{\mathbb{H}}_\wp$ can be lifted to an
idempotent in $\Ha_\wp$ by Lemma \ref{l2}. We conclude that the image
$\overline{\mathbb{H}}_i$ of every $\Ha_i$ in $\overline{\mathbb{H}}_\wp$, defined
as in last lemma, has no idempotents and therefore $\overline{\mathbb{H}}_i$ is a
field since $\overline{\mathbb{H}}_\wp$ is semisimple. It follows from Lemma
\ref{lemma34} that the image of the local generators for each $\Ha_i$ is
$\Txi_\wp(\Da_i|\Ha_i)=K_\wp^{*d_i}\oink_\wp^*$, where every $d_i$ is the dimension
of the corresponding representation. On the other hand, by Lemma \ref{l33}
 we get the
contention $\Txi_\wp(\Da|\Ha)\subseteq K_\wp^{*t_\wp}\oink_\wp^*$.
Note that
$$K_\wp^{*t_\wp}\oink_\wp^*=\prod_iK_\wp^{*d_i}\oink_\wp^*=\prod_i\Txi_\wp(\Da_i|\Ha_i)\subseteq
\Txi_\wp(\Da|\Ha)\subseteq K_\wp^{*t_\wp}\oink_\wp^*,$$ whence
equality follows. We conclude that
$\Txi_\wp(\Da|\Ha)=K_\wp^{*t_\wp}\oink_\wp^*$ as claimed.
\end{proof}

\section{Applications and examples}

%

\begin{ex}
Let $\alge=\matrici_n(K)$ and let $\mathfrak{T}_n$ be the order of
all upper triangular matrices of the form
$$\left(\begin{array}{cccc}a_1&a_2&\cdots&a_n\\0&a_1&\cdots&a_{n-1}\\
\vdots&\vdots&\ddots&\vdots\\0&0&\cdots&a_1\end{array}\right),\qquad
a_1,\dots,a_n\in\oink_K.$$ Then $\mathbb{H}_\wp$ has only one
irreducible representation of dimension $1$ for all $\wp$. It
follows that $t_\wp=1$ for all $\wp$ and therefore
$F(\mathfrak{T}_n)=K$. We conclude that every maximal order in
$\matrici_n(K)$ contains a copy of the order
$\mathfrak{T}_n$.\end{ex}

 More generally, for an arbitrary commutative suborder $\Ha$, the dimensions of the
irreducible representations of $\Ha$ are the same as the dimensions of the
irreducible representations of the maximal semi-simple suborder $\Ha_s$ of $\Ha$.
Next result follows:

\begin{cor}
If $\Ha$ is an arbitrary commutative order in the algebra $\alge$,
and $\Ha_s$ is the maximal semi-simple suborder of $\Ha$, then
$F(\Ha_s)=F(\Ha)$. In particular, if a maximal order $\Da$
contains a conjugate of $\Ha_s$, then there exists an order $\Da'$
in the spinor genus of $\Da$ satisfying $\Ha\subseteq\Da'$.
\end{cor}

Recall that the inertia degree $f_\wp(F/K)$ of an unramified abelian extension $F/K$
divides the inertia degree $f_\wp(L/K)$ of an arbitrary extension $L/K$, at every
$\wp$, if and only if $F\subseteq L$. Next result follows (compare to Theorem 4.3.4
in \cite{spinor}):

\begin{cor}\label{sse}
If $\alge$ has no partial ramification outside of $S$ and
$\Ha=\prod_i\Ha_i$, where each $\Ha_i$ is the maximal order of a
field $L_i$, then
$F(\Ha)=\Sigma\cap\left(\bigcap_iL_i\right)$.\qed
\end{cor}

\begin{ex}
If $\Ha$ is the maximal order in a semi-simple commutative algebra
of prime dimension that is not a field, then the greatest common
divisor of the degrees $[L_i:K]$ is $1$, and therefore $\Ha$ is
contained in at least one order in every spinor genus.
\end{ex}

\begin{ex}
If $\Ha=\oink_K\times\Ha_1$, then $\Ha$ is contained in at least one order in every
spinor genus.
\end{ex}

Recall that a split order in $\matrici_n(K)$ is an order
containing a copy of $\oink_K^n=\oink_K\times\cdots\times\oink_K$
\cite{split}. By applying last example to $\Ha=\oink_K^n$,  we
obtain a simple proof of the following generalization of Theorem 5
in \cite{Chevalley}.
\begin{cor}
Every spinor genera of maximal $X$-orders in $\matrici_n(K)$ contains an split
order. If $X$ is affine, then every maximal order is split.\qed
\end{cor}

An order $\Ha$ in $\alge$ is called non-selective if for every spinor genus $\Phi$
of maximal orders, there exists an order $\Da\in\Phi$ containing $\Ha$. An order
$\Ha\subseteq\alge$ is selective if it is not non-selective. This is the natural
generalization of the concept of selective order defined in \cite{FriedmannQ} and
\cite{Linowitz}. Assume now $[L:K]=p$ is a prime, and $\alge$ has no partial
ramification, so if $\mathfrak{L}$ is the maximal order of $L$, Corollary \ref{sse}
applies and $F(\mathfrak{L})=\Sigma\cap L$. Since $[L:K]$ is a prime, the order
$\mathfrak{L}$ is selective if and only if $L\subseteq\Sigma$. In this case the
order is contained in $1/p$ of all conjugacy classes. Assume this is the case and
let $\Ha\subseteq\mathfrak{L}$ be an arbitrary suborder. Then $\Ha$ is non-selective
if and only if there is a place $\wp$ that is inert for $L/K$ such that the image of
$\Ha$ in the residue field $\mathbb{L}_\wp$ coincide with the residue field
$\mathbb{K}_\wp$ of $K$, since in this case $t_\wp=1$, so $F(\Ha)$ splits at $\wp$,
and therefore $F(\Ha)=K$. Since the condition is equivalent to $\wp$ dividing the
relative discriminant of $\Ha/\oink_K$, next result follows:
\begin{cor}
Let $\alge$ be a central simple algebra without partial
ramification and let $\mathfrak{L}$ be the maximal order of a
field $L$ of prime degree over $K$. Then $\mathfrak{L}$ is
selective in $\alge$ if and only if $L\subseteq\Sigma$. In this
case a suborder $\Ha\subseteq\mathfrak{L}$ is selective unless
there is a inert place of $L/K$ dividing the relative discriminant
of $\Ha$ over $\oink_K$.
\end{cor}
Note that we no longer require that $\dim_K\alge=p^2$ in the above
result (see \cite{FriedmannQ} and \cite{Linowitz}). More
generally, the field $L$ always contains $F$, whence we have next
result:

\begin{cor}
Let $\alge$ be a central simple algebra without partial ramification and let $\Ha$
be an order contained in a field $L\subseteq\alge$. Let $N$ be the total number of
spinor genera. Then
$$\sharp\Big\{\Phi\in\overline{\mathbb{O}}\Big|\exists \Da\in\Phi\textnormal{ such that
}\Ha\subseteq\Da\Big\}=\frac1d,$$
 for some divisor $d$ of $[L:K]$.\qed
\end{cor}


\nocite{*}
\bibliographystyle{cdraifplain}


\begin{thebibliography}{99}

\bibitem{spinor}
{\scshape L.E. Arenas-Carmona}, Applications of spinor class fields: embeddings of
orders and quaternionic lattices, \textit{Ann. Inst. Fourier} \textbf{53} $\natural
7$ (2003), 2021-2038.


\bibitem{counter}
 {\scshape Luis Arenas-Carmona}. Relative spinor class
fields: A counterexample, \textit{Archiv. Math.} \textbf{91}
(2008), 486-491.

\bibitem{sheaves}
{\scshape L.E. Arenas-Carmona}, Spinor class fields for sheaves of lattices,
\textit{arXiv:1009.3280v1 [math.NT] 16 Sep 2010}.



%

\bibitem{Chevalley}
{\scshape C. Chevalley}, \textit{L'arithm\'etique sur les
alg\`ebres de matrices}, Herman, Paris, 1936.




\bibitem{FriedmannQ}
{\scshape T. Chinburg and E. Friedman}, An embedding theorem for quaternion
algebras, \textit{J. London Math. Soc.} \textbf{60} $\natural 2$ (1999), 33-44.




\bibitem{Linowitz}
{\scshape B. Linowitz and T.R. Shemanske}, Embedding orders into central simple
algebras, \textit{arXiv:1006.3683v1 [math.NT] 18 Jun 2010}.



\bibitem{Om}
{\scshape O.T. O'Meara}, \textit{Introduction to quadratic forms},
Academic press, New York, 1963.



\bibitem{reiner}
{\scshape I. Reiner},  \textit{Maximal orders}, Academic press, London, 1975.


\bibitem{split}
{\scshape T.R. Shemanske}, Split orders and convex polytopes in buildings,
\textit{J. Number Theory} \textbf{130} $\natural 1$ (2010), 101--115.


\bibitem{weil}
{\scshape A. Weil}, \textit{Basic Number Theory},
$2^{\mathrm{nd}}$ Ed., Springer Verlag, Berlin, 1973.

\end{thebibliography}
\end{document}